\documentclass[11pt]{amsart}
\usepackage{latexsym,subfigure}
\usepackage{comment,tikz}
\usepackage{amsmath,amsthm,amsfonts,amssymb,graphicx,epsfig,latexsym,color}
\usepackage{epsf,enumerate,url}

\newtheorem{thm}{Theorem}[section]
\newtheorem{lemma}[thm]{Lemma}
\newtheorem{cor}[thm]{Corollary}
\newtheorem{prop}[thm]{Proposition}
{}
{}

\theoremstyle{remark}

\newtheorem{definition}[thm]{Definition}
\newtheorem{remark}[thm]{Remark}

\newtheorem*{definition*}{Definition}
\newtheorem*{remark*}{Remark}

\newcommand{\bY}{{\bf Y}}

\renewcommand{\xi}{{\theta}}

\def\cC{{\mathcal C}}
\def\cX{{\mathcal X}}

\def\cP{{\mathcal P}}

\def\cA{{\mathcal A}}
\def\cB{{\mathcal B}}

\def\Z{{\mathbb Z}}

\def\H{{\mathbb H}}

\def\diam{\operatorname{diam}}

\begin{document}

\title{Acylindrical actions on projection complexes} \author[M. Bestvina, K. Bromberg, K. Fujiwara and A. Sisto]{Mladen
  Bestvina, Ken Bromberg, Koji Fujiwara and Alessandro
  Sisto}

  \thanks{The first two authors gratefully acknowledge the
    support by the National Science Foundation. The third author is
    supported in part by Grant-in-Aid for Scientific Research
    (No. 19340013). All four authors were supported by the National
    Science Foundation under Grant No. DMS-1440140 while the authors
    were in residence at the Mathematical Sciences Research Institute
    in Berkeley, California, during the Fall 2016 semester.}
    
\date{\today}

\begin{abstract}
 We simplify the construction of projection complexes from \cite{BBF}. To do so, we introduce a sharper version of the Behrstock inequality, and show that it can always be enforced. Furthermore, we use the new setup to prove acylindricity results for the action on the projection complexes.
 
 We also treat quasi-trees of metric spaces associated to projection complexes, and prove an acylindricity criterion in that context as well.
\end{abstract}

\maketitle

\section{Introduction}
In \cite{BBF} an axiomatic setup was given for showing that certain
groups act on quasi-trees. A typical example arises when $G =
\pi_1(\Sigma)$ where $\Sigma$ is a closed, hyperbolic surface. One
then takes a simple, closed geodesic on $\Sigma$ and lets $\bY$ be the
set of components of the pre-image of the geodesic in the universal cover
$\H^2$. Given distinct geodesics $X,Y \in \bY$ we let $\pi_Y(X)$ be the nearest
point projection of $X$ to $Y$ and observe that the diameters of these
sets will be uniformly bounded as $X$ and $Y$ vary through $\bY$.
More interestingly, if we have distinct elements $X,Y,Z \in \bY$ and
the projections $\pi_Y(X)$ and $\pi_Y(Z)$ are far apart on $Y$ then
the projections $\pi_Z(X)$ and $\pi_Z(Y)$ will be coarsely the same on
$Z$ (meaning that the diameter of the union is uniformly bounded). One
then perturbs the projections up to finite Hausdorff distance in a certain way and builds the {\em projection complex} $\cP_K(\bY)$ with vertex set
$\bY$ by fixing a large constant $K$ and adding an edge between
distinct vertices $X$ and $Z$ if the diameter of $\pi_Y(X) \cup
\pi_Y(Z)$ is at most $K$ for all $Y \in \bY\backslash\{X,Z\}$. The
central result of \cite{BBF} is that $\cP_K(\bY)$ is a quasi-tree, that is to say it is quasi-isometric
to a tree. In certain situations, the perturbation of the projections is necessary.

One of the technical challenges of \cite{BBF} is that when
the
diameter of $\pi_Y(X) \cup \pi_Y(Z)$ is large the projections
$\pi_Z(X)$ and $\pi_Z(Y)$ are coarsely equal, but are not exactly
equal. This causes problems in induction arguments, because of constants that might get worse at every step. We will see here that by assuming equality (instead of just
coarse equality) the proof that the projection complex is a quasi-tree (Theorem \ref{quasi-tree})
vastly simplifies. Unfortunately, in most naturally occurring
situations we do not have equality. In the second part of the paper we
introduce the notion of a {\em forcing sequence} and use it to show
that the projection maps can be modified (coarsely) so that we have
the desired equality (Theorem \ref{thm:forcing}). In this way one can replace the work of section
3 of \cite{BBF} with the much simpler arguments in this paper.

Besides simplifying the approach from \cite{BBF}, the new setup allows us to obtain acylindricity results. The action of a group $G$ on a metric space $\cX$ is {\em
  acylindrical} if for all $D>0$ there exist $L>0$ and $B>0$ such that
if $x,y \in \cX$ and $d_\cX(x,y) \ge L$ then there are at most $B$
elements $g\in G$ with $d_\cX(x, gx) \le D$ and $d_\cX(y, gy) \le D$. Another improvement of this paper is that, under some simple conditions, it is straightforward to show that the $G$-action on the projection complex will be acylindrical, see Theorem \ref{acylindrical}. We summarize the results described so far in the following statement (the condition for acylindricity in Theorem \ref{acylindrical} is less restrictive).

\begin{thm}[Theorems \ref{thm:forcing}, \ref{quasi-tree}, \ref{acylindrical}]
Assume that $\bY=\{(Y, \rho_Y)\}$ is a collection of metric spaces with projection distances $\{d_Y\}$ satisfying axioms (P0) - (P4) from Section \ref{sec:forcing}. After perturbing the projections as in Theorem \ref{thm:forcing}, pick a constant $K$ and consider the complex $\cP_K(\bY)$ with vertex set $\bY$ and vertices $X,Z$ connected by an edge if $d_Y(X,Z)>K$. If $K$ is sufficiently large, then $\cP_K(\bY)$ is a quasi-tree.

Moreover, if the group $G$ acts on $\bY$ preserving $d_Y$, then $G$ acts by isometries on $\cP_K(\bY)$. If in addition there exist $B$ and $N$ so that the common stabilizer of any collection of $N$ elements of $\bY$ has cardinality at most $B$, then the action is acylindrical. 
\end{thm}

In \cite{DGO}, Dahmani-Guirardel-Osin introduced the notion of a {\em hyperbolically embedded subgroup}, a generalization of the concept of a relatively hyperbolic group. To construct groups with hyperbolically embedded subgroups Dahmani-Guirardel-Osin start with the group acting on a $\delta$-hyperbolic space with a {\em WPD element} and then use the projection complexes from \cite{BBF} to construct the hyperbolically embedded subgroup. In \cite{Osin}, a converse is proven; a group with an infinite hyperbolically embedded subgroup has an  acylindrical action on a $\delta$-hyperbolic space. In \cite{Balas}, Balasubramanya improves this last theorem by showing that the $\delta$-hyperbolic space can be taken to be a quasi-tree. In section 5, we derive Balasubramanya's theorem using our methods (Theorem \ref{balas}).

As in \cite{BBF} we can also build a {\em quasi-tree of metric spaces}. In the final section we give necessary and sufficient conditions for it to be a quasi-tree or a hyperbolic space (Theorem \ref{thm.CY.quasi.tree} and Corollary \ref{CY.quasi.tree}-(2)) and, more generally, prove that it is in a natural way a tree-graded space. Furthermore, we show that (under some natural conditions) the group action on the quasi-tree of metric spaces is also acylindrical (Theorem \ref{thm:complex_acylind}).

While many of the arguments here follow a similar outline to what is in \cite{BBF} (a notable exception being the part about forcing sequences) this paper is completely self-contained, with the only exception of the use of Manning's bottleneck criterion and, in the last section, its generalization from \cite{Hume:tree_graded}, and does not require any of the results from \cite{BBF}. 
We also note that the section on hyperbolically embedded subgroups does not require the forcing sequence technology from the previous section as the projection maps defined there satisfy the equality condition without modification.

An abridged version of this paper, only containing the proof that projection complexes are quasi-trees and the perturbation of the projection distances, is available on the authors' websites. This shorter version already contains most of the ideas and techniques that we use in this paper.

\section{Axioms}
Let $\bY$ be a set and for each $Y \in \bY$ assume that we have a function 
$$d_Y\colon \bY^2 \backslash \{(Y,Y)\} \to [0, \infty]$$
such that the following {\em strong projection axioms} are satisfied
for some $\theta\geq 0$:
\begin{enumerate}[{\bf (SP 1)}]
\item $d_Y(X,Z) = d_Y(Z,X)$;
\item $d_Y(X,Z) + d_Y(Z,W) \ge d_Y(X,W)$;
\item if $d_Y(X,Z) > \theta$ then $d_Z(X,W) = d_Z(Y,W)$ for all $W\in \bY\backslash\{Z\}$;
\item $d_Y(X,X) \leq \theta$;
\item $\#\{Y|d_Y(X,Z) > \theta\}$ is finite for all $X,Z \in \bY$.
\end{enumerate}
The constant $\theta$ is the {\em projection constant}. Note that we
allow $d_Y(X,Z) = \infty$.

The most important axiom is arguably (SP 3), which is a version of the
Behrstock inequality \cite{Behrstock}. As in \cite{BBF}, we will use
it to order certain subsets of $\bY$, the idea being that if
$d_Y(X,Z)$ is large, then $Y$ is between $X$ and $Z$. We note that
(SP 3) is in fact a more precise version of the Behrstock inequality
because the conclusion is an actual equality, not an approximate
one. This allows us to know the exact value of certain $d_Y$, and it
is the key to our much simpler proofs, compared to \cite{BBF}.

\begin{lemma}\label{triples}
(SP 3) and (SP 4) imply
$$\min\{d_Y(X,Z), d_Z(X,Y)\} \leq \theta$$
\end{lemma}

\begin{proof}
If $d_Y(X,Z) > \theta$ then letting $W= Y$ in (SP 3) we have $d_Z(X,Y) = d_Z(Y,Y) \leq \theta$ by (SP 4).
\end{proof}

Define $\bY_K(X,Z)$ to be the collection of $Y \in \bY\backslash \{X,Z\}$ such that $d_Y(X,Z) > K$.

Lemma \ref{conditions} and Proposition \ref{order} below say that, for large enough $K$, $\bY_K(X,Z)$ can be totally ordered using the idea, as mentioned above, that if $d_Y(X,Z)$ is large then $Y$ is between $X$ and $Z$. The order has several equivalent characterizations, which is good for applications, and they are listed in Lemma \ref{conditions}:

\begin{lemma}\label{conditions}
For $Y_0,Y_1 \in \bY_{2\theta}(X,Z)$ the following conditions are equivalent:
\begin{enumerate}[(1)]
\item $d_{Y_0}(X, Y_1) > \theta$;
\item $d_{Y_1}(Y_0, W) = d_{Y_1}(X,W)$ for all $W \neq Y_1$;
\item $d_{Y_1}(X,Y_0) \leq \theta$;
\item $d_{Y_1}(Y_0,Z) > \theta$;
\item $d_{Y_0}(W,Y_1) = d_{Y_0}(W,Z)$ for all $W \neq Y_0$;
\item $d_{Y_0}(Y_1, Z) \leq \theta$.
\end{enumerate}
\end{lemma}

\begin{proof}
By Lemma \ref{triples},  $(1) \Rightarrow (3)$ and $(4) \Rightarrow (6)$. By (SP 2), $(3) \Rightarrow (4)$ and $(6) \Rightarrow (1)$. By (SP 3), $(1)\Rightarrow (2)$ and $(4)\Rightarrow (5)$. Since $Y_1 \in \bY_{2\theta}(X,Z)$ by letting $W=Z$ we have $(2)\Rightarrow (4)$ and similarly $(5)\Rightarrow (1)$.
\end{proof}

Given $Y_0, Y_1 \in \bY_{2\theta}(X,Z)$ we define $Y_0<Y_1$ if any one of $(1)-(6)$ hold. 

\begin{prop}\label{order}
The relation $<$ defines a total order on $\bY_{2\theta}(X,Z)$ that extends to a total order on $\bY_{2\theta}(X,Z)\cup \{X,Z\}$ with least element $X$ and greatest element $Z$. Furthermore if $Y_0<Y_1<Y_2$ then $d_{Y_1}(Y_0,Y_2) = d_{Y_1}(X,Z)$.
\end{prop}

Notice that with a coarse version of (SP 3) there would be no hope to obtain the last conclusion as stated.

\begin{proof}
By swapping $Y_0$ and $Y_1$ we see that $Y_0 < Y_1$ if and only if $Y_1 \not< Y_0$. So any two elements of $\bY_{2\theta}(X,Z)$ can be compared.

Now we check transitivity of the order. If $Y_0< Y_1$ and $Y_1<Y_2$ we apply (2) for $Y_0<Y_1$ with $W=Y_2$ and then again to $Y_1<Y_2$ with $W=Z$ to see that $d_{Y_1}(Y_0, Y_2) = d_{Y_1}(X,Y_2) = d_{Y_1}(X,Z) > 2\theta$. Applying (SP 3) and then (5) we have $d_{Y_2}(Y_0, Z) = d_{Y_2}(Y_1,Z) = d_{Y_2}(X,Z) > 2\theta$. Therefore $Y_0<Y_2$ and the total order is well defined on $\bY_{2\theta}(X,Z)$. We can extend it to a total order on $\bY_{2\theta}(X,Z) \cup \{X,Z\}$ by declaring $X$ to be the least element and $Z$ the greatest element.

Observe that we have also shown that $d_{Y_1}(Y_0,Y_2) = d_{Y_1}(X,Z)$ if $Y_0<Y_1<Y_2$.
\end{proof}

The main use of the following lemma will be to construct free groups
(for other purposes, simpler versions would suffice). It is a kind of
local-to-global principle.

\begin{lemma}\label{tree}
For any $K\geq 2\theta$ the following holds. Let $Q$ be a connected
simplicial graph and $\phi: Q^{(0)}\to \bY$ a map such that adjacent
vertices are mapped to distinct elements of ${\bf Y}$, and if $x,y$
and $z$ are distinct vertices with $x$ and $z$ adjacent to $y$ then
$d_{\phi(y)}(\phi(x), \phi(z)) > K$. Then for any immersed path
$\{x_0, \dots, x_k\}$ in $Q$ we have $\{\phi(x_0) < \cdots <
\phi(x_k)\} \subset \bY_{K}(\phi(x_0), \phi(x_k))\cup \{\phi(x_0),
\phi(x_k)\}$. In particular, $\phi$ is injective and $Q$ is a tree.
\end{lemma}

\begin{proof}
For $k\le 2$ the conclusion clearly holds. We proceed by induction on
$k$. Let $x_0, \dots, x_k$ be an immersed path and let $X_i =
\phi(x_i)$. We first show $X_i \in \bY_{K}(X_0,X_k)$ for $0<i<k$. If
$1\leq i< k-1$ then $d_{X_i}(X_0,X_{k-1}) > K$ and $d_{X_{k-1}}(X_i,X_k)
> K$ by induction. Then by (SP 3), $d_{X_i}(X_0, X_k) = d_{X_i}(X_0,
X_{k-1}) > K$ so $X_i \in \bY_{K}(X_0,X_k)$.

If $i=k-1$ we
reverse the roles of $X_0$ and $X_k$, and of $X_1$ and $X_{k-1}$.

For the order we have that $d_{X_i}(X_0,X_j) > K$ if $0<i<j< k$ since $X_i \in \bY_{K}(X_0,X_j)$ and therefore $X_i< X_j$.
\end{proof}

\begin{cor}\label{containment}
Let $K\geq 2\theta$. If $Y_0, Y_1 \in
\bY_{K}(X,Z)\cup\{X,Z\}$ and $d_Y(Y_0,Y_1) > K$ then $Y \in
\bY_{K}(X,Z)$. 
\end{cor}

\begin{proof}
We assume $Y_i\notin \{X,Z\}$, since in those cases the proof is similar and easier.

We can assume that $Y_0< Y_1$. We then apply Lemma \ref{tree} where
$Q$ is a closed interval subdivided into 4 segments with vertices
labeled $X, Y_0, Y, Y_1, Z$, in order. By assumption $d_Y(Y_0,Y_1) >K$
so we only need to check that $d_{Y_0}(X,Y)$ and
$d_{Y_1}(Y,Z)$ are greater than $K$. By (SP 3) $d_{Y_0}(X,Y) =
d_{Y_0}(X, Y_1)$. By Proposition \ref{order}, $d_{Y_0}(X,Y_1) =
d_{Y_0}(X,Z) >K$. Therefore $d_{Y_0}(X,Y) >K$.
Similarly, $d_{Y_1}(Y,Z) >K$.
\end{proof}

\section{The projection complex}
Fix $K\geq 2\theta$ and define the graph $\cP_K(\bY)$ with vertex set
$\bY$ and an edge between any two vertices $X$ and $Z$ with
$\bY_K(X,Z) = \emptyset$. We denote the distance in $\cP_K(\bY)$
simply by $d$, even though it depends on $K$.

We first note that $\cP_K(\bY)$ is connected.

\begin{lemma}\label{standard paths}
If $K\geq 2\theta$ and $X,Z\in\bY$ then $\bY_K(X,Z) \cup \{X,Z\} =
\{X<X_1< \cdots < X_k < Z\}$ is a path in $\cP_K(\bY)$.
\end{lemma}

\begin{proof}
By Corollary \ref{containment}, if $Y$ is the immediate predecessor of
$Y'$ in the total order on $\bY_K(X,Z) \cup \{X,Z\}$ then
$\bY_{K}(Y,Y') =\emptyset$ and therefore $d(Y,Y')=1$.
\end{proof}

The path $\bY_K(X,Z) \cup \{X,Z\} =
\{X<X_1< \cdots < X_k < Z\}$ is the {\it standard path} from $X$ to
$Z$. 

The following lemma says that, when moving outside the ball of radius
2 around a vertex $Z$ of $\cP_K(\bY)$, the projection to $Z$ varies
slowly, where slowly is independent of $K$.

\begin{lemma}\label{two}
If $K\geq 2\theta$ then the following holds. Let $X_0, X_1, Z \in \bY$ with $d(X_0,X_1) = 1$ and $d(X_0,Z) \ge 2$. Then $d_Z(X_0, X_1) \leq \theta$.
\end{lemma}

\begin{proof}
Since $d(X_0,Z) \ge 2$ there exists a $Y \in \bY_K(X_0,Z)$ and therefore by (SP 3) $d_Z(X_0,X_1) = d_Z(Y,X_1)$. If $d_Z(Y,X_1) = d_Z(X_0,X_1) > \theta$ then by (SP 3) $d_Y(X_0,X_1) = d_Y(X_0,Z) > K$, a contradiction with $\bY_K(X_0,X_1) = \emptyset$.
\end{proof}

The following lemma and its corollary are the key to proving that $\cP_K(\bY)$ is a quasi-tree. They say that, when moving outside the ball of radius 3 around a vertex $Z$ of $\cP_K(\bY)$, the projection to $Z$ basically does not change.

\begin{lemma}\label{induction}
If $K\geq 3\theta$ then the following holds. Let $X_0, \dots, X_k$ be a path in $\cP_K(\bY)$  and $Z \in \bY$ with $d(X_i,Z) \ge 3$. Then greatest elements of $\bY_{3\theta}(X_i,Z)$ all agree.
\end{lemma}

\begin{proof}
We can assume $k=1$.  Let $Y_0$ and $Y_1$ be the corresponding
greatest elements and assume they are distinct. By Corollary
\ref{containment}, $\bY_{3\theta}(Y_i,Z) =\emptyset$ so $d(Y_i,Z)=1$
and $d(X_i,Y_i) \ge 2$. Applying Lemma \ref{two} we see that
$d_{Y_i}(X_0, X_{1}) \leq \theta$ and therefore by (SP 2),
$d_{Y_i}(X_{1-i}, Z) > 2\theta$. In particular, both $Y_0$ and $Y_1$
are in $\bY_{2\theta}(X_i,Z)$ for $i=0,1$. 
We can assume that
$Y_0<Y_1$ in $\bY_{2\theta}(X_0,Z)$. By Lemma \ref{conditions}(6) this
means that $d_{Y_0}(Y_1,Z)\leq\theta$ and so we also have $Y_0<Y_1$ in
$\bY_{2\theta}(X_1,Z)$. 
In particular, $d_{Y_1}(X_0,
Z) = d_{Y_1}(Y_0,Z) = d_{Y_1}(X_1,Z)> 3\theta$, contradicting the
assumption that $Y_0$ is the greatest element of
$\bY_{3\theta}(X_0,Z)$.
\end{proof}

\begin{cor}\label{three}
If $K\geq 3\theta$ then the following holds. Let $X_0, \dots, X_k$ be a path in $\cP_K(\bY)$ and $Z \in\bY$ with $d(X_i,Z)\ge 3$. Then $d_Z(X_i, X_j) \leq \theta$ for all $i,j$.
\end{cor}

\begin{proof}
By Lemma \ref{induction}, there exists a $Y \in \bY$ that is the
greatest element of all of the $\bY_{3\theta}(X_i, Z)$. We now have
$d_Z(X_i,X_j)=d_Z(X_i,Y)\leq\theta$ by (SP 3) and Lemma \ref{triples}.
\end{proof}

We can now use Manning's {\it bottleneck condition} \cite{Manning.quasi.tree} to show that
$\cP_K(\bY)$ is a quasi-tree. We will use a variant of Manning's
condition that is described in \cite{BBF}: Let $X$ be a connected
simplicial graph with its usual combinatorial metric and $D\geq
0$. Assume that for all vertices $v_0,v_1 \in X^{(0)}$ there is a path
$p$ such that the $D$-Hausdorff neighborhood of any path from $v_0$ to
$v_1$ contains $p$. Then $X$ is a quasi-tree.

\begin{thm}\label{quasi-tree}
For $K \ge 3\theta$, $\cP_K(\bY)$ is a quasi-tree.
\end{thm}

\begin{proof}
By Lemma \ref{standard paths}
$\bY_K(X,Z) \cup \{X,Z\} =
\{X<X_1< \cdots < X_k < Z\}$ is a path in $\cP_K(\bY)$.  Let $X=Y_0,
Y_1, \dots, Y_n=Z$ be an arbitrary path from $X$ to $Z$. Since
$d_{X_i}(X,Z) > K \ge 3\theta$ by Corollary \ref{three} there must
be a $Y_j$ such that $d(Y_j,X_i) \le 2$.
 Therefore $\cP_K(\bY)$
satisfies the bottleneck condition and is a quasi-tree.
\end{proof}

The following lemma is a variant of \cite[Lemma 3.9]{HO} proved by
Hull and Osin in the context of hyperbolically
embedded subgroups.

\begin{lemma}\label{divide}
If $K > 2\theta$ then the following holds. Given $X,Y,Z \in \bY$ the
union $\bY_K(X,Y) \cup \bY_K(Y,Z)$ contains all but at most two
elements of $\bY_K(X,Z)$, and if there are two such elements they are
consecutive.

In particular, $\bY_K(X,Z)$ is written as the disjoint union of three
consecutive segments (some possibly empty) so that the initial segment (if not empty)
is also initial in $\bY_K(X,Y)$, the second contains at most two
elements, and the terminal segment (if not empty) is also terminal in
$\bY_K(Y,Z)$.

\end{lemma}

\begin{proof}
Set $\bY_K(X,Z) = \{X_1 < \cdots < X_k\}$. Then for any $X_i$ at least
one of $d_{X_i}(X, Y)$ or $d_{X_i}(Y,Z)$ is $> \theta$.
 If it is the
former then $X_j \in \bY_K(X,Y)$ for all $j<i$ (proof: use (SP 3)
twice to see that $d_{X_j}(X,X_i)>K$ implies
$d_{X_i}(X_{j},Y)=d_{X_i}(X,Y)>K\geq\theta$ so
$d_{X_{j}}(X,Y)=d_{X_{j}}(X,X_i)>K$) 
while if it is the
latter then $X_j \in \bY_K(Y,Z)$ for all $j>i$.

Now assume that $X_i$ is the smallest element not in the union. By the previous paragraph $d_{X_{i+1}}(X,Y) \leq \theta$, and in turn $d_{X_{i+1}}(Y,Z) > \theta$
by triangle inequality. But this implies that $X_j \in \bY_K(Y,Z)$ if $j>i+1$.

The rest is clear. See Figure \ref{tripod}.
\end{proof}

\begin{figure}[h]
\begin{center}
  \scalebox{0.5}{\input{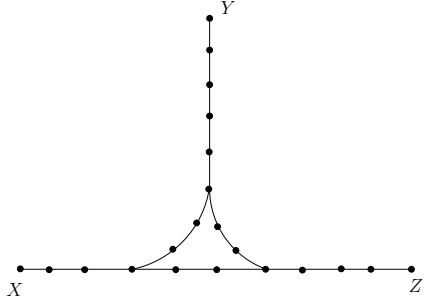}}
    \caption{A typical triangle of standard paths}
    \end{center}
    \label{tripod}
\end{figure}

\begin{cor}\label{distance}
If $K\geq 3\theta$ then the following holds. Let $X\neq Z$ and $n
= |\bY_K(X,Z)|+1$. Then $\lfloor \frac n2 \rfloor +1 \le d(X,Z) \le
n$.
\end{cor}

\begin{proof}
The second inequality follows from the fact that
$\bY_K(X,Z)\cup\{X,Z\}$ is path from $X$ to $Z$. The proof of the
other inequality is by induction on $d(X,Z)$, the case $d(X,Z)=1$ being
clear.
 
 Suppose $d(X,Z)\geq 2$. Pick $Y\neq X,Z$ on a geodesic from $X$ to $Z$. Setting $n_1=|\bY_K(X,Y)|+1, n_2= |\bY_K(Y,Z)|+1$, Lemma \ref{divide} gives $n_1+n_2\geq n-1$. Using the induction hypothesis,
  we get $d(X,Z)=d(X,Y)+d(Y,Z)\geq\lfloor\frac{n_1}{2}\rfloor + \lfloor\frac{n_2}{2}\rfloor +2\geq \lfloor\frac{n}{2}\rfloor+1 $. 
\end{proof}

\begin{cor}
Standard paths are quasi-geodesics.
\end{cor}

Assume that a group $G$ acts on $\bY$ and the functions $d_Y$ are $G$-invariant. Then $G$ acts isometrically on $\cP_K(\bY)$.

The following theorem gives a simple criterion for the action on $\cP_K(\bY)$ to be acylindrical, in terms of finiteness of the size of the stabilizer of several elements of $\bY$. Roughly speaking, we look at far away $X,Z$ and an element $g$ that moves both a small distance, and deduce that a large middle interval in $\bY_K(X,Z)$ is also contained in $\bY_K(gX,gZ)$. With too many $g$, we would get too many elements stabilizing several elements of $\bY$.

\begin{thm}\label{acylindrical}
If $K\geq 3\theta$ then the following holds. Assume that for some
fixed $N$ and $B$, for any $N$ distinct elements of any $\bY_K(X,Z)$ the
common stabilizer is a finite subgroup of size at most $B$. Then the
action of $G$ on $\cP_K(\bY)$ is acylindrical.
\end{thm}

\begin{figure}[h]
\begin{center}
  \scalebox{0.5}{\input{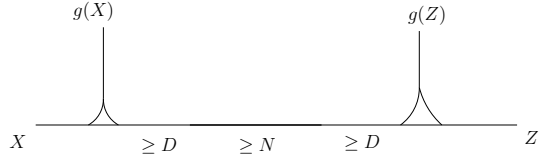}}
    \caption{Proof of acylindricity}
    \end{center}
    \label{quadrilateral}
\end{figure}

\begin{proof}
  Fix $D>0$ and assume that $X,Z \in \bY$ with $d(X,Z) \ge N + 4D
  +6$. Let $g\in G$ be such that $d(X,gX)\leq D$ and $d(Z,gZ)\leq
  D$. Consider the quadrilateral of standard paths spanned by
  $X,Z,gX,gZ$. Let the segment $I$ ( and $J$) in $\bY_K(X,Z)$
  be obtained by removing initial and terminal segments of length $2D+2$ (and $D+2$, respectively). Then $I$
  has length $|I|\geq N$ and $J$ has length $|I|+2D$, and by Lemma
  \ref{divide} $J\subset \bY_K(gX,gZ)$. Thus $g(I)\subset J$ and there
  are $\leq 2D+1$ possible restrictions of $g$ to $I$ (translation by
  a number in $[-D,D]$). By the pigeon-hole principle and the
  assumption about stabilizers it follows that there are at most
  $B(2D+1)$ such elements $g$.
\end{proof}

In the following theorem we construct free groups acting on $\cP_K(\bY)$. We use Lemma \ref{tree} to certify that certain elements generate a free group.

\begin{thm}\label{freegroup}
Assume that $K \ge 3\theta$ and fix $Y \in \bY$. If there exists
$g_0,g_1, g_2\in G$ such $d_Y(g_iY, g_jY) > K$ and $d_Y(g_i^{-1}Y, g_j^{-1}Y) >
K$ when $i\neq j$ then
the group generated by $g_1g_0^{-1}$ and $g_2g_1^{-1}$ is free and
acts faithfully on $\cP_K(\bY)$ with orbit map a QI-embedding.
\end{thm}

\begin{proof}
Let $F \subset G$ be the subgroup generated by $g_1g_0^{-1}$ and
$g_2g_1^{-1}$. Consider the theta graph $\Theta$ with two vertices $v_0,v_1$
and three oriented edges labeled $g_0,g_1,g_2$ connecting $v_0$ to
$v_1$. Then $\pi_1(\Theta)$ can naturally be identified with the free
group on $g_1g_0^{-1}$ and $g_2g_1^{-1}$ and there is a canonical
epimorphism $\psi:\pi_1(\Theta)\to F$. Let $\Gamma$ be the covering space
of $\Theta$ corresponding to the kernel of $\psi$. Thus $F$ acts on
$\Gamma$ as the deck group. The edges of
$\Gamma$ have an
induced orientation and labeling by the $g_i$'s.

Define the $F$-equivariant map $\phi:\Gamma^{(0)}\to \bY$ by sending a
base vertex $w_0$ to $Y$ and if $u_1u_2\cdots u_k$ is a path from
$w_0$ to $w$ with $u_i\in \{g_0^{\pm 1},g_1^{\pm 1},g_2^{\pm 1}\}$
(with exponents necessarily alternating) then $\phi(w)=u_1u_2\cdots
u_k(Y)$.  Our assumption implies that the hypotheses of Lemma
\ref{tree} are satisfied and so $\Gamma$ is a tree and $\psi$ is an
isomorphism. The last sentence follows from Corollary \ref{distance}.
\end{proof}

In practice it is easy to verify the conditions of Theorem
\ref{freegroup}. In applications of the projection complex the set
$\bY$ is a collection of infinite diameter metric spaces and the
subgroup that fixes each metric space acts coarsely transitively. It
is then relatively easy to find the necessary elements of $G$.

\section{Forcing sequences}\label{sec:forcing}
Let $\bY= \{(Y, \rho_Y)\}$ be a collection of metric spaces.

Consider the following axioms from \cite{BBF}. For all pairwise
distinct $X,Y,Z\in{\bf Y}$, for some $\theta\geq 0$ and distance
functions $d^\pi_Y(X,Z)$:

\begin{enumerate}
\item [{\bf (P0)}] \label{item:not_transverse} 
$d^\pi_Y(X,X) \le \theta$.
 \item [{\bf (P1)}] \label{item:Behrstock}  if $d^\pi_Y(X,Z)> \theta$ then $d^\pi_X(Y,Z) \leq \theta$.
 \item [{\bf (P2)}] \label{item:finite} $\{W\neq X,Z: d^\pi_W(X,Z)> \theta\}$ is finite.
 \item[{\bf (P3)}]
$d^\pi_Y(X,Z) = d^\pi_Y(Z,X)$
\item[{\bf (P4)}]
$d^\pi_Y(X,Z) + d^\pi_Y(Z,W) \ge d^\pi_Y(X,W)$
\end{enumerate}

Families of metric spaces with projection maps satisfying 
(P0)-(P4) occur naturally in many contexts. See the introduction to \cite{BBF} for some examples.
In most cases there are projections $\pi_Y(Z)\subseteq Z$ so that
$$d^\pi_Y(X,Z):=\diam \pi_Y(X)\cup\pi_Y(Z)$$ satisfy (P0)-(P4).

The goal of this section is to prove the following theorem.

\begin{thm}\label{thm:forcing}
Assume that $\bY=\{(Y, \rho_Y)\}$ is a collection of metric spaces
with $\{d_Y^\pi \}$ satisfying (P0) - (P4) with constant $\theta$.
Then there are $\{d_Y \}$ satisfying (SP 1) -(SP 5) for the constant $11 \theta$ such that 
$$d_Y^\pi - 2 \theta \le d_Y \le d_Y^\pi + 2 \theta.$$

Moreover,  assume that $d_Y^\pi$ are obtained 
from projections $\pi_Y$. Then there are
projections $\pi'_Z(X)$ for $X\neq Z$ satisfying the following:
\begin{enumerate}[(1)]
\item $\pi'_Z(X)\subseteq N_\theta(\pi_Z(X))$,
\item $d_Y(X,Z)$ is equal to 
the diameter of $\pi'_Y(X)\cup\pi'_Y(Z)$.
\item If a group $G$ acts on $\bY$ preserving the metrics and the
  projections $\pi_Z(X)$, then $G$ also preserves the new projections
  $\pi'_Z(X)$.
  
  Thus $G$ acts on a quasi-tree as in Theorem
  \ref{quasi-tree} and the action is frequently acylindrical as in
  Theorem \ref{acylindrical}.
\end{enumerate}
\end{thm}

Mimicking the earlier section we let $\bY^\pi_K(X,Z)$ be the
collection of $Y \in \bY\backslash\{X,Z\}$ such that $d^\pi_Y(X,Z) >
K$.

\subsection{Modifying the distance $d^\pi$}

We assume $d_Y^\pi$ satisfy (P0) - (P4).

The first step is to modify $d^\pi$ to achieve monotonicity (see Lemma
\ref{preinsertion}). Recall from \cite{BBF} that for $X\neq Z$ we define
${\mathcal H}(X,Z)$ as the set of pairs $(X',Z')\in \bY\times\bY$ such
that one of the following holds.
\begin{itemize}
\item $d_X^\pi(X',Z'),d_Z^\pi(X',Z')>2\theta$,
\item $X=X'$ and $d_Z^\pi(X,Z')>2\theta$,
\item $Z=Z'$ and $d_X^\pi(X',Z)>2\theta$,
\item $X=X'$ and $Z=Z'$.
\end{itemize}

\begin{lemma}
  If $d_Y^\pi(X,Z)>2\theta$ and $(X',Z')\in {\mathcal H}(X,Z)$ then,
  after possibly permuting $X'$ and $Z'$,
  $d_Y^\pi(X,X'),d_Y^\pi(Z,Z')\leq\theta$. In particular,
  $|d_Y^\pi(X,Z)-d_Y^\pi(X',Z')|\leq 2\theta$.
\end{lemma}

\begin{proof}
  By the triangle inequality (P4)
  $$d_X^\pi(X',Y)+d_X^\pi(Y,Z') \ge d_X(X',Z') > 2\theta$$ and
  therefore $\max\{d_X^\pi(X',Y), d_X^\pi(Y,Z')\} >\theta$. After
  possibly permuting $X'$ and $Z'$ we can assume $d_X^\pi(X',Y)
  >\theta$ and therefore by (P1) we have $d_Y^\pi(X,X') \le
  \theta$. By another application of the triangle inequality
  $$d_Y^\pi(X,X') + d^\pi_Y(X',Z) \ge d_Y^\pi(X,Z)>2\theta$$
  and since $d_Y^\pi(X,X') \le\theta$ this implies that $d_Y^\pi(X',Z)
  >\theta$. Therefore (P1) implies that $d_Z^\pi(X',Y) \le\theta$. Now,
  replacing $X$ with $Z$ in the above application of the triangle
  inequality (P4) we have $\max\{d_Z^\pi(X',Y), d_Z^\pi(Y,Z')\} >\theta$
  and therefore $d_Z^\pi(Y,Z') >\theta$ since we have just seen that
  the other term is $\le\theta$. Another application of (P1) gives
  that $d_Y^\pi(Z,Z') \le \theta$.

  The final inequality follows from the triangle inequality (P4).
\end{proof}

\begin{cor}\label{inclusions}
  If $d_Y^\pi(X,Z)>4\theta$ then ${\mathcal H}(X,Z)\subseteq {\mathcal
    H}(X,Y)$.
\end{cor}

\begin{proof}
  Suppose $(X',Z')\in {\mathcal H}(X,Z)$. We will again assume the
  first bullet holds and leave the other cases to the reader. To show
  $(X',Z')\in {\mathcal
    H}(X,Y)$ it suffices to argue that $d_Y^\pi(X',Z')>2\theta$, and
  this follows from $d_Y^\pi(X,Z)>4\theta$ and the lemma.
\end{proof}

We now define the modified distance
$$\tilde d_Y(X,Z)=\sup_{(X',Z')\in{\mathcal H}(X,Z)}d_Y^\pi(X,Z)$$ if
$d^\pi_Y(X,Z)>2\theta$, and $\tilde d_Y(X,Z)=2\theta$ otherwise. Thus
$$d_Y^\pi(X,Z)\leq \tilde d_Y(X,Z)\leq d_Y^\pi(X,Z)+2\theta$$

The triangle inequality for $\tilde d$ holds only up to an error of $2\theta$.
What we gain with this modification is the following monotonicity
property.

\begin{lemma}\label{preinsertion}
If $\tilde d_Y(X,Z)>5\theta$ and $\tilde d_W(Y,Z)>7\theta$ then
$\tilde d_Y(X,W)\geq
\tilde d_Y(X,Z)$.
\end{lemma}

\begin{proof}
We have $d^\pi_W(Y,Z)>5\theta$ so $d^\pi_Y(W,Z)\leq\theta$. Likewise,
$d^\pi_Y(X,Z)>3\theta$ so $d^\pi_Y(X,W)\geq
d^\pi_Y(X,Z)-d^\pi_Y(W,Z)>2\theta$ and so $d^\pi_W(X,Y)\leq
\theta$. Thus $d^\pi_W(X,Z)\geq
d^\pi_W(Y,Z)-d^\pi_W(X,Y)>4\theta$. Corollary \ref{inclusions} gives
${\mathcal H}(X,W)\supseteq {\mathcal H}(X,Z)$. We saw above that
$d^\pi_Y(X,W)>2\theta$ and the statement follows.
\end{proof}

\subsection{The second (and final) modification of $d^\pi$}
To prove the theorem we need to modify the $d^\pi_Y$ so that
they satisfy the projection axioms (SP 1)-(SP 5). The key notion to do so
is that of a forcing sequence, which uses the first modification
$\tilde d$. In slightly different contexts, an idea similar to forcing
sequences has appeared in \cite{bartels} and \cite[Section 2.B]{BB}.

Also, if $d_Y^\pi$ are obtained from  projections
$\pi_Y$, this modification is realized by modifications 
of $\pi_Y$ to $\pi'_Y$.
\begin{definition}
A \emph{$K$-forcing sequence} is a sequence $Y=Y_0,\dots,Y_n=Z$
of distinct elements of $\bY$ so that
$\tilde d_{Y_i}(Y_{i-1},Y_{i+1}) > K$ (for all $i=1, \dots,n-1$).
\end{definition}

Notice that if $X\neq Z$ then $X=Y_0,Y_1=Z$ is a (usually non-maximal)
forcing sequence.

\begin{lemma}\label{pretermination}
  Let $Y_0,Y_1,\cdots,Y_n$ be a $4\theta$-forcing sequence. Then
  \begin{enumerate}[(i)]
  \item $d^\pi_{Y_n}(Y_0,Y_{n-1})\leq\theta$,
    \item
      $|d^\pi_{Y_j}(Y_i,Y_k)-d^\pi_{Y_j}(Y_{j-1},Y_{j+1})|\leq2\theta$
      for all $i<j<k$.
  \end{enumerate}
\end{lemma}

\begin{proof} We prove (i) by induction on $n$, starting with the
  obvious case $n=1$. Suppose that it is true for a given $n$
 and let us prove it for $n+1$. Observe that
 $d^\pi_{Y_i}(Y_{i-1},Y_{i+1})>2\theta$.

Since $d^\pi_{Y_n}(Y_0,Y_{n-1})\leq \theta$, by the triangle inequality
we have $d^\pi_{Y_n}(Y_0,Y_{n+1})>\theta$, so that
$d^\pi_{Y_{n+1}}(Y_0,Y_{n})\leq \theta$, as required.

To prove (ii) note that both $Y_i, Y_{i+1}, \dots, Y_{j}$ and $Y_{k},
Y_{k-1}, \dots, Y_j$ are $4\theta$-forcing sequences. We apply (i) to
each of them and (ii) then follows from the triangle inequality.
\end{proof}

The lemma below tells us when we can insert elements in forcing
sequences, and it will be used to show that if $\tilde d_W(X,Z)$ is large,
then any maximal forcing sequence from $X$ to $Z$ goes through
$W$. Its proof uses the monotonicity of $\tilde d$.

\begin{lemma}\label{lem:insert}
Let $Y_0,\dots,Y_n$ be a $K$-forcing sequence with $K\geq 7\theta$ and
$W \in \bY$ such that $\tilde d_W(Y_i, Y_{i+1}) > K$. Then $Y_0, \dots,
Y_i, W, Y_{i+1}, \dots, Y_n$ is a $K$-forcing sequence.
\end{lemma}

\begin{proof}
We need to argue that $\tilde d_{Y_i}(Y_{i-1},W),
\tilde d_{Y_{i+1}}(W,Y_{i+2})>K$. Both follow from Lemma
\ref{preinsertion}, e.g. $\tilde d_{Y_i}(Y_{i-1},W)\geq
\tilde d_{Y_i}(Y_{i-1},Y_{i+1})>K$.
\end{proof}

\begin{lemma}\label{lem:refine}
For $K\geq 7\theta$, any $K$-forcing sequence from $X$ to $Z$ can be
refined into a maximal one.
\end{lemma}

\begin{proof}
The obvious process of refinement, using Lemma \ref{lem:insert}, must
terminate by Lemma \ref{pretermination}(ii) and (P2).
\end{proof}

\begin{lemma}\label{lem:BGI}
Let $Y_0, \dots, Y_n$ be a maximal $K$-forcing sequence, $K\geq
7\theta$, and let $W \in \bY$ with $d^\pi_W(Y_0,Y_n)> K+2\theta$. Then
$W= Y_i$ for some $i\in\{1, \dots, n-1\}$.
\end{lemma}

\begin{proof}
We assume that $W$ is distinct from all the $Y_i$ and derive a contradiction.

By Lemma \ref{pretermination}, $d^\pi_{Y_i}(Y_0, Y_n) > 2\theta$. We
first observe that if $d^\pi_W(Y_i, Y_n) > \theta$ then
$d^\pi_{Y_{i}}(Y_0, W) \ge d^\pi_{Y_{i}}(Y_0, Y_n) -
d^\pi_{Y_{i}}(W,Y_n) > \theta$ since by (P1) $d^\pi_{Y_{i}}(W,Y_n)\leq
\theta$. Again applying (P1) we have $d^\pi_W(Y_0, Y_i) \leq \theta$.

We have proved that for every $i$, either $d^\pi_W(Y_0,Y_i)\leq\theta$
or $d^\pi_W(Y_n,Y_i)\leq\theta$. There is some $i$ such that
$d^\pi_W(Y_0,Y_i)\leq\theta$ and
$d^\pi_W(Y_n,Y_{i+1})\leq\theta$. From the triangle inequality and our
assumption, we have $d^\pi_W(Y_i,Y_{i+1})>K$ and therefore $Y_0, \dots, Y_{i-1}, W,
Y_i,\dots, Y_n$ is a $K$-forcing sequence by Lemma
\ref{lem:insert}, contradicting our maximality assumption.
\end{proof}

\begin{lemma}\label{lem:concatenate}
Let $Y_0,\dots,Y_{n-1},Y_n$ be a maximal $K$-forcing sequence with
$K\ge 7\theta$ and suppose that $X\in\bY$ satisfies
$d^\pi_{Y_0}(X,Y_n) > K+\theta$. Then there exists a maximal
$K$-forcing sequence from $X$ to $Y_n$ with penultimate element
$Y_{n-1}$.
\end{lemma}
 
\begin{proof}
By Lemma \ref{pretermination}, $d^\pi_{Y_0}(Y_1, Y_n) \leq \theta$
so $$\tilde d_{Y_0}(X,Y_1)\geq d^\pi_{Y_0}(X,Y_1) \ge
d^\pi_{Y_0}(X,Y_n) - d^\pi_{Y_0}(Y_1,Y_n) > K.$$ Therefore
$X,Y_0,\dots, Y_n$ is a $K$-forcing sequence. Any maximal refinement
will have the required property for if in the refinement an element
appeared between $Y_{n-1}$ and $Y_n$ then the original sequence would
not be maximal.
\end{proof}

\begin{definition}[Penultimate elements]\label{defn:second_perturbation}
For distinct elements $X,Z\in{\bf Y}$  define  a subset $P_Z(X)=\{W\}
 \subset {\bf Y}$,
 where $W$ are  all penultimate elements  of maximal $7\theta$-forcing sequences from $X$ to $Z$. 
 Note that $P_Z(X)$ is not empty.

When $Y\neq X,Z$, we define
 $$d_Y(X,Z) = \sup d_Y^\pi(W_1,W_2),$$
 where $W_1 \in P_Y(X), W_2 \in P_Y(Z)$.
 
 If projection maps $\pi_Z$ are defined, set 
 $$\pi'_Z(X)=\bigcup \pi_Z(W),$$
 where $W \in P_Z(X)$.

\end{definition}
 Note that $d_Y(X,Z)$ is equal to 
  the $\rho_Y$-diameter of $\pi'_Y(X)\cup\pi'_Y(Z)$.
  In other words, $d_Y = d_Y^{\pi'}$.
  
  Also, if a group $G$ acts on ${\bf Y}$ preserving the 
  metrics and the projections $\pi_Z(X)$, then $G$
  preserves $\pi'_Z(X)$ by the construction.

\begin{lemma}\label{lem:second_perturbation}
We have
$$d_Y^\pi - 2 \theta \leq d_Y\leq d_Y^\pi+2\theta.$$
Also, 
$$\pi'_Z(X)\subseteq N_\theta(\pi_Z(X)),$$
where $N_\theta$ denotes the Hausdorff
$\theta$-neighborhood. 

\end{lemma}

\begin{proof}
By Lemma \ref{pretermination} if $W$
is the penultimate element of a $7\theta$-forcing sequence from $X$ to
$Z$ then $d^\pi_Z(X,W) \leq\theta$. The inequalities follow
from triangle inequality of $d_Y^\pi$.
The second claim is clear.
\end{proof}

\begin{lemma}\label{lem:stable}
If $d_Y(X,Z) > 11\theta$ then $P_Z(X)=P_Z(Y)$,
hence $\pi'_Z(X) =\pi'_Z(Y)$.
\end{lemma}

\begin{proof}
By Lemma \ref{lem:second_perturbation} if $d_Y(X,Z) > 11\theta$ then
$d^\pi_Y(X,Z) > 9\theta$. By Lemma \ref{lem:BGI} if $X=Y_0, \dots,
Y_n=Z$ is a maximal $7\theta$-forcing sequence then $Y=Y_i$ for some
$i \in \{1, \dots, n-1\}$. Then $Y_i, \dots, Y_n$ is a maximal
$7\theta$-forcing sequence from $Y$ to $Z$ and it follows that $P_Z(Y)
\supseteq P_Z(X)$.

By Lemma \ref{lem:concatenate} any maximal $7\theta$-forcing sequence
from $Y$ to $Z$ can be extended to a maximal $7\theta$-forcing
sequence from $X$ to $Z$ with the same penultimate element so
$P_Z(X) \supseteq P_Z(Y)$.

We showed $P_Z(X)=P_Z(Y)$.
Then $\pi'_Z(X) =\pi'_Z(Y)$ follows. 
\end{proof}

\begin{prop}\label{axioms follow}
If $d_Y^\pi$ satisfy (P0)-(P4), then 
$d_Y$ satisfy the projection axioms (SP 1)-(SP 5) with projection constant
  $11\theta$.
\end{prop}

\begin{proof}
(SP 1) and (SP 2) are trivial. 
(SP 3) is exactly Lemma \ref{lem:stable} and (SP 5) follows from
  (P2) and Lemma \ref{lem:second_perturbation}. The
  other axioms are clear.
\end{proof}

Lemma \ref{lem:second_perturbation} and Proposition \ref{axioms
  follow} complete the proof of Theorem \ref{thm:forcing}.

\section{Acylindrical examples}\label{sec.acyl}

In this section we apply Theorem \ref{acylindrical} to prove in concrete examples that the action on one of the projection complexes we constructed is acylindrical.

\subsection{Hyperbolically embedded subgroups}
 Let $G$ be a group and $H$ a subgroup. Fix a (possibly infinite) set
 $S \subset G$ such that $S \cup H$ generates $G$. Let $\Gamma(G,
 S\sqcup H)$ be the Cayley graph for this generating set; more precisely, we introduce double edges corresponding to elements in $S\cap H$ and regard every edge as labelled by the corresponding copy of a generator. We define a
 function $\hat d\colon H\times H \to [0,\infty]$ as follows. If $x,y
 \in H$ are connected by a path in $\Gamma(G, S\sqcup H)$ that does
 not contain any edges from $H$ we let $\hat d(x,y)$ be the length of
 the shortest such path. If there is no such path we let $\hat d(x,y)
 = \infty$. Then $H$ is {\em hyperbolically embedded} in $G$ (with
 respect to the generating set $S$) if
 \begin{enumerate}
 \item $\Gamma(G, S\sqcup H)$ is $\delta$-hyperbolic;
 \item $\hat d$ is proper.
 \end{enumerate}
 
 Each coset $aH$ in $\Gamma(G,S\sqcup H)$ consists of the vertices of a
 complete graph and when we refer to $aH$ as a subset of $\Gamma(G,
 S\sqcup H)$ we refer to this complete graph whose edges are labeled
 by the elements of $H$. A path $p$ in
 $\Gamma(G,S\sqcup H)$ {\em penetrates} the coset $aH$ if the
 intersection of $p$ with $aH$ contains a segment. Note that  every
 coset has diameter 1 so a geodesic that penetrates a coset will
 intersect it in exactly one segment.
 
 The following is a consequence of \cite[Proposition 4.13]{DGO}, but we provide a proof in the interest of self-containment.
 
  \begin{lemma}\label{ngon}
 There exists a $C>0$ such that the following holds. Suppose that we
 have a geodesic quadrilateral in $\Gamma(G, S\sqcup H)$ with sides
 $s, p_0,p_1,p_2$ so that $s$ is an edge in the coset $aH$ with
 endpoints $s_0,s_1$, and no $p_i$ penetrates $aH$.  Then $\hat
 d(s_0,s_1) \le C$.
 \end{lemma}
 
 \begin{proof}
We label the sides so that $p_2$ is opposite to $s$, and $s_i\in p_i$.
 Let $\delta$ be an integer so that $\Gamma(G,S\sqcup H)$ is
 $\delta$-hyperbolic. Recall that for a geodesic quadrilateral, the
 $2\delta$-neighborhood of any three sides contains the fourth. 
 
We consider the case when the lengths of $p_0$ and $p_1$ are
$>2\delta+2$, leaving the other (easier) cases to the reader. Let
$x_i$ be the vertex of $p_i$ at distance $2\delta+2$ from the endpoint
of $p_i$ that belongs to $aH$, and let $y_i$ be a vertex at distance
$\leq 2\delta$ from $x_i$ on a side of the quadrilateral distinct from
$p_i$. Note that necessarily $y_i\not\in s$ and further a geodesic
$[x_i,y_i]$ does not intersect $aH$.

If both $y_i$ belong to $p_2$ we have the path
$s_0,x_0,y_0,y_1,x_1,s_1$ made of segments in the quadrilateral and
geodesics $[x_i,y_i]$. It has length bounded by a function of $\delta$
and it is disjoint from $aH$ except at the endpoints.

If say $y_0\in p_1$, we have the path $s_0,x_0,y_0,s_1$ with the same
conclusion. 
 \end{proof}

 Given subsets of vertices $X$ and $Y$ define $\pi_X(Y)$ to be the set of all $x\in X$ so that $x$ is an endpoint of a geodesic that minimizes the distance between $X$ and $Y$.
 
 \begin{lemma}\label{penetrate}
 Let $aH$, $bH$ and $cH$ be distinct cosets. If every geodesic 
 that minimizes the distance between $aH$ and $cH$ penetrates $bH$ then $\pi_{cH}(aH) = \pi_{cH}(bH)$.
 \end{lemma}
 
 \begin{proof}
 Let $p$ be a geodesic that minimizes the distance from $aH$ to $cH$ and that penetrates $bH$. In particular $p$ contains a single segment in $bH$. Decompose $p$ into three segments $p=p_0 p_1 p_2$ where $p_1$ is the segment in $bH$. Then $p_0$ is a geodesic that minimizes the distance from $aH$ to $bH$ and $p_2$ minimizes the distance from $bH$ to $cH$. In particular the terminal endpoint of $p$ will lie in $\pi_{cH}(bH)$ so $\pi_{cH}(aH) \subseteq \pi_{cH}(bH)$. 
 
 Given any other point $z \in \pi_{cH}(bH)$ we can find a minimizing geodesic $p'_2$ from $bH$ to $cH$ that has terminal endpoint $z$. Let $p'_1$ be a segment in $bH$ that has the same initial endpoint as $p_1$ and whose terminal endpoint is the initial endpoint of $p'_2$. Then $p_0p'_1p'_2$ is a path from $aH$ to $cH$ that has the same length as $p$ (because $p$ penetrates $bH$) and is therefore a minimizing geodesic. Therefore $z \in \pi_{cH}(aH)$ and $\pi_{cH}(bH) \subseteq \pi_{cH}(aH)$. 
 \end{proof}
 
 \begin{lemma}\label{diameter}
 If $aH\neq bH$ and $x,x' \in \pi_{aH}(bH)$ then $\hat d(x,x') \le C$,
 where $C$ is the constant from Lemma \ref{ngon}.
 \end{lemma}

 \begin{proof}
 Let $p$ and $p'$ be geodesics that minimize the distance from $aH$ to $bH$ and have initial endpoints $x$ and $x'$, respectively. Connect the terminal endpoints of $p$ and $p'$ with segments to form a 4-gon. Then Lemma \ref{ngon} implies that $\hat d(x,x') \le C$.
 \end{proof}

The above lemma shows that there is a coarsely well-defined projection
$\Gamma\to aH$. Using Lemma \ref{ngon} it is easy to see that
geodesics that do not penetrate $aH$ have uniformly bounded image in
$aH$ with respect to the $\hat d$-metric, a version of the Bounded
Geodesic Image Theorem.
 
 Let
 $$d_{bH}(aH, cH) = \underset{x\in \pi_{bH}(aH), z\in\pi_{bH}(cH)} \sup \hat d(x,z).$$
 \begin{prop}\label{axiomshold}
 The collection of cosets $\{aH\}$ and the functions $d_{aH}$ satisfy
 the projection axioms (SP 1)-(SP 5) with $\theta =3C+1$.
 \end{prop}
 
 \begin{proof}
 Both (SP 1) and (SP 2) are clear and (SP 4) follows from Lemma \ref{diameter}. 
 
 For (SP 3), assume $d_{bH}(aH,cH)\geq 3C+1$ and let $p$ be a geodesic
 minimizing the distance between $aH$ and $cH$ with initial endpoint
 $x \in aH$ and terminal endpoint $z \in cH$. We will show that $p$
 penetrates $bH$ and then (SP 3) follows from Lemma
 \ref{penetrate}. Assume not and let $q_0$ be a geodesic that
 minimizes the distance from $x$ to $bH$ and let $x'$ be the terminal
 endpoint of $q_0$. We can assume that $x' \in \pi_{bH}(aH)$ for if
 not there will be a geodesic from $\pi_{aH}(bH)$ to $\pi_{bH}(aH)$
 that is shorter than $q_0$ and we can connect this geodesic to $x$
 with a segment in $aH$ to form a path from $x$ to $bH$ that is at
 most as long as $q_0$. As $q_0$ is minimizing this new path must be a
 geodesic. Similarly we can find a geodesic $q_2$ from $z$ to a point
 $z' \in \pi_{bH}(cH)$ that minimizes the distance between $z$ and
 $bH$. Note that neither $q_0$ nor $q_2$ have segments in $bH$. We
 then let $q_1$ be the segment in $bH$ between $x'$ and $z'$ and
 $p^{-1}q_0q_1q_2$ is a 4-gon that satisfies the conditions of Lemma
 \ref{ngon} so $\hat d(x',z') \le C$. But since $d_{bH}(aH,cH) \ge
 3C+1$ we have $\hat d(x',z') \ge C+1$ by Lemma \ref{diameter}, a contradiction.
 
 Given cosets $aH, bH$ and $cH$, by the previous paragraph if
 $d_{bH}(aH,cH) \ge 3C+1$ then every geodesic that minimizes the
 distance between $aH$ and $cH$ penetrates $bH$. Since any geodesic
 can only penetrate a finite number of cosets this proves (SP 5).
 \end{proof}
 
 The group $G$ acts on the set of cosets $\{aH\}$ by $g(aH) =
 (ga)H$. The coset $aH$ is fixed by the subgroup $H^a = aHa^{-1}$. 
 We will need the following result of Dahmani-Guirardel-Osin, which we prove for the convenience of the reader.
 
 \begin{prop}[\cite{DGO}]\label{boundedintersection}
 There is a uniform bound on $|H^a \cap H^b|$ over all distinct cosets $aH$, $bH$.
 \end{prop}
 
 \begin{proof}
 We can assume $b=1$. Also, up to multiplying $a$ on the left by an
 element of $H$, we can assume that a geodesic $\gamma$ in $\Gamma(G, S \sqcup H)$
 from $1$ to $a$ does not intersect $H$ except
 for $1$.
  If $h\in H\cap H^a$ then there exists a quadrilateral where two
  opposite sides are $H$-translates of $\gamma$, one side consists of
  the edge in $H$ from $1$ to $h$, and a side is contained in $aH$. By
  Lemma \ref{ngon}, $\hat d(1,h) \le C$, and we are done by local
  finiteness of $\hat d$. 
 \end{proof}

 This implies that the stabilizer of two distinct cosets will have
 uniformly bounded size.

 The following theorem then
 follows from Theorems \ref{quasi-tree} and \ref{acylindrical}.
 
 \begin{thm}[{Balasubramanya {\cite{Balas}}}]\label{balas2}
 Suppose $G$ contains an infinite hyperbolically embedded subgroup
 $H$ of infinite index. Then  $G$
 has a non-elementary acylindrical action on the quasi-tree $\cP_K(\{aH\})$, where $H$ is a hyperbolically embedded infinite subgroup of $G$ of infinite index, $K$ is large enough, and the projections are defined above.
 \end{thm}
 
 \begin{proof}
 Since we have Proposition \ref{axiomshold}
 and Proposition \ref{boundedintersection},
 it follows from Theorems \ref{quasi-tree} and \ref{acylindrical} that
 $G$ has an acylindrical action on the quasi-tree $\cP_K(\{aH\})$. 
 
 To see that the action is non-elementary, we observe that $H$ acts transitively on
 itself and since $H$ is infinite and $\hat d$ is proper for
 any $K \ge 0$ we can find $h_0=1,h_1,h_2\in H$
 such that $\hat d(h_i,h_j)
 > K$, $\hat d(h_i^{-1}, h_j^{-1}) >K$.
 
 Now choose $s\in S \backslash H$, where 
 $S$ is the set that is used for $\Gamma(G,S \sqcup H)$.
 Set $g_i=h_i s h_i$.
 Then $h_i^\pm \in \pi_{H} (g_i^\pm H)$.
 And we can apply Theorem
 \ref{freegroup} to $H$ and $g_0,g_1,g_2$.
 \end{proof}
 
 A group $G$ is \emph{acylindrically hyperbolic} if it acts acylindrically and non-elementarily on a hyperbolic space, and it is proven in \cite{Osin} that this is equivalent to containing an infinite hyperbolically embedded subgroup of infinite index.
 The following theorem immediately from Theorem \ref{balas2}
 since the action on the projection complex $\cP_K(\{aH\})$ has only one vertex
 in the quotient.
 
 \begin{thm}[\cite{Balas}]\label{balas}
 If $G$ is acylindrically hyperbolic then $G$
 has a co-bounded,  non-elementary acylindrical action on a quasi-tree.
 \end{thm}

\subsection{Mapping class groups}\label{sec.mcg.acyl}

In this subsection we assume that the reader is familiar with the theory of curve complexes and subsurface projections, as developed in \cite{MasurMinsky1,MasurMinsky2}.
Let $\Sigma$ be closed connected oriented surface with finitely many punctures, and supporting a finite-area hyperbolic metric. We will consider the collection $\bY=\{\cC(Y)\}$ of all curve complexes of isotopy classes of subsurfaces $Y$ of $\Sigma$ obtained cutting $\Sigma$ along a non-separating simple closed curve. Two such subsurfaces $Y,Z$ overlap if and only if they are not isotopic, so that, just as in \cite[Page 6]{BBF}, in view of of results in \cite{MasurMinsky2} and \cite{Behrstock} we have that $\bY$ with subsurface projections satisfies axioms (P0)-(P2).

\begin{thm}
 For $\bY$ as above, $MCG(\Sigma)$ acts acylindrically and
 non-elementarily on $\cP_K(\bY)$ for all sufficiently large $K$.
\end{thm}

\begin{proof}
 By Theorem \ref{acylindrical}, it suffices to show that, for some sufficiently large $K$, if $\bY_K(X,Z)$ contains 4 distinct elements $\cC(X_0)<\cC(Y_0)<\cC(Y_1)<\cC(X_1)$ then the common stabilizer of $X_0,Y_0,Y_1,X_1$ is finite (finite subgroups of $MCG(\Sigma)$ have bounded cardinality). Since the stabilizer of the (isotopy class of the) subsurfaces we are considering coincides with the stabilizer of (the isotopy class of) either of its boundary components, it suffices to show that $\partial X_0,\partial Y_0,\partial Y_1,\partial X_1$
 fill the subsurface (if $K$ is large
 enough). Consider any essential simple closed curve $c$, and let us
 show that it intersects the boundary of one of the subsurfaces. Up to
 switching $X$ and $Z$ and re-indexing, we can assume that $c$ is not
 parallel to $\partial Y_0$, so that $c$ has a well-defined subsurface
 projection to $\cC(Y_0)$.  Since $\partial X_0$ and $\partial Y_1$
 have far away subsurface projection to $\cC(Y_0)$, $c$ must intersect
 one of them. The action is nonelementary by
 Theorem \ref{freegroup}; the details are left to the reader. 
\end{proof}

\begin{remark} In general, the action of $MCG(\Sigma)$ on standard projection
  complexes is not acylindrical. For example, say $\Sigma$ has genus 5
  and consider the action of $MCG(\Sigma)$ on $\cP_K(\bY)$ where $\bY$
  is the collection of genus 3 subsurfaces with 1 boundary
  component. Choose a nonseparating curve $a$ on $\Sigma$. Then the
  Dehn twist in $a$ and its powers fix all subsurfaces in the
  complement of $a$, so it suffices to show that the set $\bY_a$ of
  elements of $\bY$ disjoint from $a$ form an unbounded set. Choose
  $Y\in\bY_a$ and $f\in MCG(\Sigma)$ that fixes $a$ and is
  pseudo-Anosov on $\Sigma\smallsetminus a$ and so that the distance
  in $Y$ between the stable and unstable laminations of $f$ is large
  compared to $K$. Then the set $\{f^N(Y)\mid N\in\Z\}$ is unbounded
  in $\cP_K(\bY)$ ($f$ acts as a loxodromic isometry).
\end{remark}

\subsection{WPD elements and $B$-contracting geodesics}

We only sketch this application. For details, the reader is directed
to \cite[Section 4]{BBF.quasi.tree}.

Let $X$ be a geodesic metric space. Assume that a group
$G$ acts on $X$ by isometries. Let $f\in G$ be a hyperbolic element
(i.e. the translation length is positive). For convenience, we will
also assume that $f$ acts as a translation on a geodesic line
$\gamma\subset X$. Assume further that $\gamma$ is {\it strongly
  contracting}, i.e. the image under the nearest point projection
$p:X\to\gamma$ (which is in general a multivalued map) of any metric
ball disjoint from $\gamma$ has uniformly bounded diameter. This
implies that there is a subgroup $EC(f)$, which is virtually cyclic,
such that if $g\in EC(f)$ then $g(\gamma)$ and $\gamma$ have finite
Hausdorff distance, and if $g\not\in EC(f)$ then $p(g(\gamma))$ has
uniformly bounded diameter. For convenience we will assume that
$EC(f)$ leaves $\gamma$ invariant. Both assumptions made for
convenience can be removed, at the expense of making the definitions
below more complicated, or else replacing $X$ by a quasi-isometric
space where these assumptions hold.

Let $\bY$ be the set of $G$-translates of $\gamma$. For $A,B,C\in\bY$
define $d^\pi_B(A,C)=\diam p_B(A)\cup p_B(C)$ where $p_B:X\to B$ is
the nearest point projection to $B$.

\begin{thm}\label{qtree.acyl}
The set $\bY$ and the functions $d_B^\pi, B \in \bY$
satisfy (P0) - (P4) for some $\theta>0$.
Hence, $d_Y$ satisfy (SP 1) - (SP 5) for $10 \theta$.
Thus for a sufficiently large $K$,
$\cP_K(\bY)$ is a quasi-tree on which $G$ acts.
Moreover, the action is acylindrical.
\end{thm}

\section{Quasi-trees of metric spaces}
We now return to the setup of Section \ref{sec:forcing} with a
collection of metric spaces $\bY = \{(Y, \rho_Y)\}$, projections
$\pi_Y$ with (P 0') and metrics $d_Y$, obtained from 
$\pi_Y$,  satisfying axioms (P0)-(P4). 

Here, we will
make the extra assumption that 
{\bf metric spaces are graphs with each edge
having length one}. 

\subsection{Construction}
Using Theorem \ref{thm:forcing} and Lemma \ref{lem:stable}
we can modify the
projections (and suitably increase $\theta$) to replace (P1) with

\medskip
\noindent
{\bf (P1')} if $d_Y(X,Z) > \theta$ then $\pi_X(Y) = \pi_X(Z)$.

\medskip
The functions $d_Y$ then satisfy the projection axioms (SP 1)-(SP 5).

As in \cite{BBF} we build the quasi-tree of metric spaces $\cC_K(\bY)$
by taking the union of the metric spaces in $\bY$ with an edge of
length $K>0$ 
connecting every pair of points in $\pi_Y(X)$ and
$\pi_X(Y)$ if $d_{\cP_K(\bY)}(X,Y) = 1$. For convenience, we will
assume that $\theta$ and $K$ are integers.

We define a metric $\rho$ (that is possibly infinite) on the disjoint
union of elements of $\bY$ by setting $\rho(x_0,x_1) =
\rho_X(x_0,x_1)$ if $x_0,x_1\in X$, for some $X\in\bY$, and
$\rho(x_0,x_1) = \infty$ if $x_0$ and $x_1$ are in different spaces in
$\bY$. Assume that the group $G$ acts isometrically on $\bY$ with this
metric and that the projections $\pi_X$ are $G$-invariant,
i.e. $\pi_{gX}(gY) = g(\pi_X(Y))$. Then $G$ acts isometrically on
$\cC_K(\bY)$. We will give conditions for this action to be
acylindrical.

In what follows we will adopt the convention that
{\bf  lower case letters
will refer to vertices in $\cC_K(\bY)$ with the corresponding upper
case letter denoting the metric space in $\bY$ that contains the
vertex}.  

It will be convenient to extend the definition of the projections
$\pi_Y$ to vertices in $\cC_K(\bY)$. For a vertex $x \in \cC_K(\bY)$
we set $\pi_Y(x) = \pi_Y(X)$ if $X\neq Y$. If $X=Y$ then $\pi_Y(x) =
\{x\}$. We then set $d_Y(x,z) = \diam\pi_Y(x)\cup \pi_Y(z)$ for $x, z
\in \cC_K(\bY)$. 
We also define
$$\bY_L(x,z) = \{Y \in \bY | d_Y(x,z) > L\}$$
and observe that it is possible for $X$ or $Z$ to be in $\bY_L(x,z)$.

To save notation, when $x$ and $z$ are vertices of $\cC_K(\bY)$, we denote by $d_\cC(x,z)$ their distance in $\cC_K(\bY)$.

First of all, we prove a coarsely Lipschitz property of projections:

\begin{lemma}\label{complex_lipschitz}
Assume that $K>\theta$. Let $x,z$ be vertices of $\cC_K(\bY)$ and let
$Y\in\bY$.  If $d_\cC(x,z) \ge \theta$ then $d_Y(x,z) \le
d_\cC(x,z)$. If $d_\cC(x,z) \leq \theta$ then $d_Y(x,z)\le\theta$.
\end{lemma}

\begin{proof}
If $d_\cC(x,z) \le \theta$ then $X=Z$ so $d_Y(x,z)=d_\cC(x,z)\leq
\theta$ if $Y=X$ and $d_Y(x,z)\leq\theta$ by (SP 4) otherwise.  

In
general, we induct on the the distance because our spaces are graphs. If $d_\cC(x,z) \ge \theta+1 >
\theta$ let $x_0$ be a vertex adjacent to $x$ such that $d_\cC(x,z) =
d_\cC(x,x_0) + d_\cC(x_0,z)$. If $X=X_0$ then $d_\cC(x,x_0) = 1$ so
$d_\cC(x_0,z) \ge d_\cC(x,z) -1 \ge \theta$ since $d_\cC(x,z) \ge
\theta +1$. Furthermore, since $X=X_0$ either $d_Y(x,z) = d_Y(x_0,z)$
if $X\neq Y$ or $d_Y(x,z) \le 1+ d_Y(x_0,z)$ if $X=Y$. In both cases
$d_\cC(x,z) = d_\cC(x_0,z)+1 \ge d_Y(x_0,z)+1 \ge d_Y(x,z)$.

If $X\neq X_0$ then $d_\cC(x,x_0) = K$ and it may be that
$d_\cC(x_0,z) < \theta$. 
However, in this case the vertex adjacent to
$z$ will be in $Z$ and we can apply the previous case,
unless $x_0=z$. But if $x_0=z$, then $d_\cC(x,z)=K$
and $d_Y(x,z) \le \theta$, therefore this is fine too. 

So we can
assume that $d_\cC(x_0,z) \ge \theta$ and $d_Y(x_0,z) \le
d_\cC(x_0,z)$. Since $x$ and $x_0$ are adjacent we have that
$d_Y(x,x_0) \le K$
 and therefore $d_\cC(x,z) = K + d_\cC(x_0,z) \ge
d_Y(x,x_0) + d_Y(x_0,z) \ge d_Y(x,z)$.
\end{proof}

\begin{lemma}\label{complex_divide}
Assume $K> 2\theta$. Given $x,y,z \in \cC_K(\bY)$ with $\bY_K(x,z)
\cup \{X,Z\} = \{X=X_0< X_1<\cdots < X_n =Z\}$ there exists $k \in
\{0,\dots, n\}$ such that if $i< k$ then $\pi_{X_i}(y) = \pi_{X_i}(z)$
and if $i >k+1$ then $\pi_{X_i}(y) = \pi_{X_i}(x)$.
\end{lemma}

\begin{proof}
Let $k$ be the smallest value such that $\pi_{X_k}(y) \ne
\pi_{X_k}(z)$. If there is no such $k$, or $k\in\{n,n-1\}$, then we
are done by setting $k=n-1$, so we now assume that $k$ exists and $k<n-1$. 

First we
observe that $d_{X_{k+1}}(X_k, y) \leq \theta$ for otherwise
$\pi_{X_k}(y) =\pi_{X_k}(X_{k+1}) = \pi_{X_k}(z)$ by (P1'). But then, by the
triangle inequality, we have $d_{X_{k+1}}(y,X_{i})  \ge
d_{X_{k+1}}(X_i,X_k) - \theta \ge \theta$ for all $i > k+1$. 
Therefore, by (P1'),  $\pi_{X_i}(y) =
\pi_{X_i}(X_{k+1}) = \pi_{X_i}(x)$.
\end{proof}

\begin{thm}\label{thm:complex_dist}
If $K \ge 4\theta$ then for all $x,z\in\cC_K(\bY)$ we have
$$\frac14 \sum_{Y\in\bY_K(x,z)} d_Y(x,z) \le d_\cC(x,z) \le 2\sum_{Y\in\bY_K(x,z)} d_Y(x,z) + 3K.$$
\end{thm}

\begin{proof}
Let $\bY_K(x,z) \cup \{X,Z\} = \{X=X_0< X_1< \cdots< X_n =Z\}$. We first prove the upper bound by finding a path from $x$ to $z$. Fix points $x_i \in \pi_{X_i}(x) = \pi_{X_i}(X_{i-1})$ and $z_i \in \pi_{X_i}(z) = \pi_{X_i}(X_{i+1})$. Note that $x=x_0$ and $z=z_n$. We then have $d_\cC(x_i, z_i) \le d_{X_i}(x_i,z_i) \le d_{X_i}(x,z)$. Since $d_{\cP_K(\bY)}(X_i,X_{i+1}) = 1$ we also have that $d_\cC(z_i,x_{i+1}) = K$. Therefore
$$d_\cC(x,z) \le  \sum_i d_{X_i}(x,z) + nK.$$
If $X_0$ (or $X_n$) are not in $\bY_K(x,z)$ then $d_{X_0}(x,z) < K$ (or $d_{X_n}(x,z) \le K$). On the other hand for $X_i \in \bY_K(x,z)$ we have $d_{X_i}(x,z) + K \le 2d_{X_i}(x,z)$. The upper bound follows.

The lower bound is more involved. Let $x=y_0,y_1, \dots, y_k=z$ be a geodesic. Fix $i_1, \cdots, i_{n-1}$ such that $\pi_{X_{j-1}}(y_{i_j}) = \pi_{X_{j-1}}(z)$ but $\pi_{X_{j-1}}(y_{i}) \neq \pi_{X_{j-1}}(z)$ if $i< i_j$. Note that if $\pi_{X_j}(y_i) = \pi_{X_j}(z)$ then $d_{X_j}(X_{j-1}, y_i) \ge \theta$ so $\pi_{X_{j-1}}(y_i) = \pi_{X_{j-1}}(X_j) = \pi_{X_{j-1}}(z)$. Therefore $i_j \le i_{j+1}$. (However, it is possible that $i_j = i_{j+1}$.)

Next we show that $\pi_{X_{j+3}}(y_{i_j}) = \pi_{X_{j+3}}(x)$. For this we observe that $\pi_{X_{j-1}}(y_{i_j -1}) \neq \pi_{X_{j-1}}(z)$ so by Lemma \ref{complex_divide} we have that $\pi_{X_{j+1}}(y_{i_j-1}) = \pi_{X_{j+1}}(x)$. This implies that $\pi_{X_{j+1}}(y_{i_j}) \neq \pi_{X_{j+1}}(z)$ for if it did we would have $d_{X_{j+1}}(y_{i_j-1}, y_{i_j}) =d_{X_{j+1}}(x, z) \ge K$ contradicting that $y_{i_j-1}$ and $y_{i_j}$ are consecutive vertices in a geodesic. Another application of Lemma \ref{complex_divide} implies that $\pi_{X_{j+3}}(y_{i_j}) = \pi_{X_{j+3}}(x)$.

By Lemma \ref{complex_lipschitz}
$$d_\cC(y_{i_{j-3}}, y_{i_{j+1}}) \ge d_{X_j}(y_{i_{j-3}}, y_{i_{j+1}}) = d_{X_j}(x,z)$$
and therefore
$$d_\cC(x,z) \ge \sum_id_{X_{j+4i}}(x,z)$$
where $j=0,1,2$ or $3$. Summing over $j$ gives the lower bound.
\end{proof}

\subsection{Acylindricity}
\begin{thm}\label{thm:complex_acylind}
Let $K\geq 4\theta$. Assume that for each $Y\in\bY$ the stabilizer of $Y$ acts acylindrically on $Y$ with uniform constants independent of $Y$. Furthermore assume that for some fixed $N$ and $B$, for any $N$ distinct elements of $\bY$ the common stabilizer is a finite subgroup of size at most $B$. Then the action of $G$ on $\cC_K(\bY)$ is acylindrical.
\end{thm}

\begin{proof}
Fix $D\ge \theta>0$. By assuming that $D\ge\theta$, by Lemma \ref{complex_lipschitz} we have that if $d_\cC(x,x') \le D$ then $d_Y(x,x') \le D$ for all $Y\in \bY$. We will use this repeatedly throughout the proof.

By Theorem \ref{acylindrical}, $G$ acts on $\cP_K(\bY)$ acylindrically so there exists $L_\cP>0$ and $B_\cP>0$ such that if $d_\cP(X,Z)>L_\cP$ then there at most $B_\cP$ elements $g \in G$ such that $d_\cP(X, gX), d_\cP(Z, gZ) < D$. By our assumption, there exist $L_\bY$ and $B_\bY$ such that for every $Y \in \bY$ and $x,z \in Y$ with $d_Y(x,z) \ge L_\bY$ there are at most $B_\bY$ elements $g\in G$ in the stabilizer of $Y$ such that $d_Y(x,gx), d_Y(z,gz) \le 2D$. It will be convenient to assume that $L_\bY \ge K$.

Fix $X,Z \in \bY$ and $x \in X$ and $z \in Z$. Note that it is possible that $X=Z$. Let $\cA = \{g \in G| d_\cC(x, gx) \le D \mbox{ and } d_\cC(z, gz) \le D\}$.
Using the distance formulas, Corollary \ref{distance} and Theorem \ref{thm:complex_dist}, we have that there exists an $L_\cC$ such that if $d_\cP(x,z) \ge L_\cC$ then either:
\begin{enumerate}
\item $d_\cP(X,Z) \ge L_\cP$ or

\item there exists a $Y \in \bY_{L_\bY+2D+4\theta}(x,z)$ and $|\bY_K(x,z)| \le 2L_\cP$.
\end{enumerate}
Since the natural projection $\cC_K(\bY)\to\cP_K(\bY)$ is 1-Lipschitz if $g \in \cA$ then $d_\cP(X, gX) \le D$ and $d_\cP(Z, gZ) \le D$. Therefore if (1) holds there at most $B_\cP$ elements in $\cA$.

Now assume (2) holds. For any $g \in G$ we have that $d_Y(x,y) \le d_Y(x, gx) + d_Y(gx,gz) + d_Y(z, gz)$. Therefore if $g \in \cA$ we have $d_Y(gx,gz)  \ge L_\bY +4\theta$. In particular, $Y \in \bY_{L_\bY + 4\theta}(gx,gz) \subset \bY_K(gx,gz)$. Now let $\cA_i$ be the set of all $g \in \cA$ such that $gY$ is the $i$th element of $\bY_K(gx,gz)$. Since $\bY_K(x,y)$ contains at most $2L_\cP$ elements, if $i > 2L_\cP$ then $\cA_i$ is empty. We will see that each $\cA_i$ has at most $B_\bY$ elements and therefore $\cA$ has at most $2L_\cP B_\bY$ elements.

Fix $g \in \cA_i$ and pick an $x' \in \pi_Y(gx)$ and $z'\in \pi_Y(gz)$. Let $$\cB = \{h \in G| h(Y) = Y, d_Y(x',hx') \le 2D \mbox{ and } d_Y(z', hz')\le 2D\}.$$ 
Then by our assumption, there is a constant $B_\bY$ that does 
not depend on $Y$ such that 
$|\cB| \le B_\bY$.

If $g' \in \cA_i \subset \cA$ then $d_Y(gx, g'x) \le d_Y(x,gx)+d_Y(x,g'x) \le 2D$. We also have $d_Y(x', g'g^{-1}x') \le d_Y(gx, g'x) \le 2D$ since $x' \in \pi_Y(gx)$ and $g'g^{-1}x' \in g'g^{-1}(\pi_Y(gx)) = \pi_{g'g^{-1}Y}(g'x)=\pi_Y(g'Y)$. Similarly $d_Y(z', g'g^{-1}z') \le 2D$ and therefore $g'g^{-1} \in \cB$. This gives the desired bound on the size of $\cA_i$.
\end{proof}

\subsection{Tree-gradedness}

In this section we study the geometry of $\cC_K(\bY)$. In particular, we prove that it is a quasi-tree (resp. hyperbolic) when the elements of $\bY$ are uniform quasi-trees (resp. uniformly hyperbolic), provided that $K$ is large.

\begin{lemma}\label{lem:far_from_proj}
 Assume $K\geq 4\theta$. 
 \\
 (1)
 Let $x,z\in \cC_K(\bY)$ be vertices connected by an edge, and let $Y\in\bY$ be so that $d_{\cC}(x,\pi_Y(x))>3K$. Then $\pi_Y(z)=\pi_Y(x)$.
 \\
 (2)
 Let $x,z\in\cC_K(\bY)$ be vertices with  $d_Y(x,z)> \theta$
 for $Y \in \bY$, then 
 any path from $x$ to $z$ intersects the (closed) neighborhood of radius $3K$
 around $\pi_Y(x)$ (and also $\pi_Y(z)$).
\end{lemma}

\begin{proof}
(1) 
 By Theorem \ref{thm:complex_dist},  there exists $W\in Y_K(x,\pi_Y(x))$ (and necessarily $W\neq Y$). By Lemma \ref{complex_lipschitz} we have $W\in Y_{3\theta}(z,\pi_Y(x))$. Hence we have $\pi_Y(x)=\pi_Y(W)=\pi_Y(z)$, where each equality holds either because $x$ and/or $z$ lie in $W$ or because of (P1').  
 
 (2) By (P0'), we have $\pi_Y(x)\not=\pi_Y(z)$ therefore any path from $x$ to $z$ intersects the (closed) neighborhood of radius $3K$ by (1).
\end{proof}

Let $x,z\in\cC_K(\bY)$. Let 
$\{X=X_0 < X_1 < \cdots <X_{n}=Z\}$ 
be the standard path in $\cP_K(\bY)$ from Lemma \ref{standard paths},
as in the proof of Theorem \ref{thm:complex_dist}. We also use
the notation $x_i, z_i$ from there. 
A {\it standard path} from $x$ to $z$ is a path 
joining $x=x_0, z_0, x_1, z_0, \cdots, x_n, z_n=z$ in this order
such that between $x_i,z_i$ it is a geodesic in $\cC(Y_i)$
and that between $z_i, x_{i+1}$ it is an edge. 

The following theorem also follows from Theorem \ref{thm:tree_graded}
below, but we give a separate proof since Theorem
\ref{thm:tree_graded} uses results from the literature. Recall that we
discussed (a variation of) the bottleneck property right before
Theorem \ref{quasi-tree}.

\begin{thm}\label{thm.CY.quasi.tree}
Let $K \ge 4\theta$.
Suppose every $Y$ satisfies the bottleneck property 
with uniform constant $D$. Then $\cC_K(\bY)$ is a quasi-tree. 
\end{thm}

\begin{proof}
First, we note that one may assume that each $Y$ satisfies
the (variant of the) bottleneck property for a geodesic between any given two points
with respect to a uniform constant which is maybe larger than $D$, 
but we keep using $D$.

Let $x,z \in\cC_K(\bY)$ be any points. We will check the variant of the bottleneck property
 for a standard path between $x,z$. 
 Let $[x_i,z_i]$ be the part (a geodesic) of the standard path in $X_i$.
Let $\gamma$ be any path from $x$ to $z$.
We want to show $[x_i,z_i]$ is contained in a Hausdorff neighborhood of $\gamma$.

For simplicity assume $X_i \not= X_0, X_n$.
By Lemma \ref{lem:far_from_proj} (2) applied to $x,z,X_i$, the path $\gamma$ intersects each 
$4K$-ball around $x_i,z_i$ (we added a $K$ to account for the diameter of $\pi_{X_i}(x)$, $\pi_{X_i}(z)$). Let $w_1, \cdots, w_m$ be the vertices in a 
subpath of $\gamma$ starting in the $4K$-ball at $x_i$
and ending in the $4K$-ball at $z_i$.
Set $y_j=\pi_{X_i}(w_j)$ then $y_1, \cdots, y_m$
is a (coarse) path in $X_i$, starting in the $4K$-ball around
$x_i$ and ending in the $4K$-ball around $z_i$
by Lemma \ref{complex_lipschitz}.
(Here a coarse path means that $|y_j - y_{j+1}| \le \theta$.)
By the bottleneck property of $X_i$, 
$[x_i,z_i]$ is contained in the $L$-Hausdorff neighborhood of 
the path $\{y_i\}$, where $L$ depends only on $K, D$.

But we now show that the path $\{y_j \}$ is contained in the $4K$-Hausdorff
neighborhood of the path $\{ w_i\}$.
To see that fix a vertex $y_j$. If $d_{X_i}(w_1,w_j) > \theta$, 
apply Lemma \ref{lem:far_from_proj} (2)  
to $w_1, w_j, X_i$,  then $y_j$ is contained in the $4K$-neighborhood
of the path between $w_1,w_j$.
Otherwise $d_{X_i}(w_j, w_m) > \theta$, then 
we apply the lemma to $w_j,w_m$ and we are done too. 

In conclusion, $[x_i,z_i]$ is contained in the $(4K+L)$-neighborhood
of $\gamma$.
We are left with the case that $X_i=X_0$ (or $X_n$).
But if $|x_0-z_0| \le  2 \theta$ then $[x_0,z_0]$ is 
contained in the $2\theta$-neighborhood of $\gamma$, or 
the argument is same as above.
\end{proof}

We now observe that $\cC_K(\bY)$ is a \emph{tree-graded space}. This notion was introduced in \cite{DrutuSapir} where tree-graded spaces arise as asymptotic cones of relatively hyperbolic groups, but a simpler example (which is more relevant for us) are Cayley graphs of free products $A*B$ with respect to a generating set contained in $A\cup B$, which are tree-graded with respect to the copies of the Cayley graphs of $A$ and $B$ that they contain.

A geodesic metric space $X$ is said to be tree-graded with respect to the collection of geodesic subspaces $\mathcal P$, called \emph{pieces}, if distinct elements of $\mathcal P$ intersect in at most one point, and every simple loop in $X$ is contained in some $P\in\mathcal P$.

\begin{thm}\label{thm:tree_graded}
 Let $K\geq 4\theta$. Then there exists $C$ so that $\cC_K(\bY)$ is $(C,C)$--quasi-isometric to a tree-graded space each of whose pieces is $(C,C)$--quasi-isometric to some $Y\in\bY$.
\end{thm}

\begin{proof}
 Using Lemma \ref{lem:far_from_proj}, 
 one can prove the relative bottleneck property from \cite{Hume:tree_graded} just as in \cite[Proposition 2.8]{Hume:tree_graded}, so that the conclusion follows from \cite[Theorem 1]{Hume:tree_graded}.
\end{proof}

We now collect some immediate consequences of the theorem and elementary properties of tree-graded spaces.

\begin{cor}\label{CY.quasi.tree}
 Let $K\geq 4\theta$. Then:
 \begin{enumerate}
  \item If, for some $C$, each $Y\in\bY$ is $(C,C)$--quasi-isometric to a tree, then $\cC_K(\bY)$ is a quasi-tree.
  \item If, for some $\delta$, each $Y\in\bY$ is $\delta$--hyperbolic, then $\cC_K(\bY)$ is hyperbolic.
  \item $\cC_K(\bY)$ is hyperbolic relative to $\bY$ (more precisely, the family of copies of the $Y\in\bY$ that it contains).\footnote{In the terminology of \cite{DrutuSapir}, $\cC_K(\bY)$ is asymptotically tree-graded with respect to $\bY$.}
 \end{enumerate}
\end{cor}

\begin{proof}
All properties follow from the analogous statements about tree-graded spaces, as outlined below.

 It is readily checked that if the pieces of a tree-graded space $X$ satisfy the bottleneck property uniformly, then the same holds for the whole space. In fact, consider a geodesic $\gamma$ in $X$ and a path $\alpha$ with the same endpoint. Up to replacing $\alpha$ with a path whose image is contained in $\alpha$, we can assume that $\alpha$ is injective. The conclusion that $\gamma$ is contained in a uniform neighborhood of $\alpha$ now easily follows from the fact that any simple loop consisting of a subpath of $\gamma$ and a subpath of $\alpha$ is contained in a piece.
 
 The proof that if the pieces of a tree-graded space $X$ are hyperbolic then $X$ is hyperbolic follows from a similar argument, where $\alpha$ is now a concatenation of two geodesics.
 
 Finally, tree-graded spaces are hyperbolic relative to their pieces by \cite[Theorem 3.30]{DrutuSapir}.
\end{proof}

\subsection{Examples}
Using Theorem\ref{thm:complex_acylind} 
we briefly discuss the acylindricity of the action 
on quasi-trees of metric spaces for  examples in 
\ref{sec.acyl}.

First, we use the same setting as in the section \ref{sec.mcg.acyl}.
Notice that action of $MCG(\Sigma)$ on $\cC_K(\bY)$ is not acylindrical because the the action of the stabilizer of any $\cC(Y)\in\bY$ has infinite kernel (generated by a Dehn twist around a boundary component of $Y$).

Next, in the same setup as in Theorem \ref{qtree.acyl},
again without assuming that an axis exists for $f$ that is 
preserved by $EC(f)$,
we also obtain $\cC_K(\bY)$, 
on which $G$ acts.
By  Corollary \ref{CY.quasi.tree}, $\cC_K(\bY)$ is a quasi-tree.
 
This action is also acylindrical.
Although $f$ fixes a point $\gamma$ in $\cP_K(\bY)$, 
$f$ is a hyperbolic isometry on $\cC_K(\bY)$
with an axis $\gamma$ ($\gamma$ is a subset in 
$\cC_K(\bY)$ that is invariant by $f$.
Moreover $\gamma$ is a geodesic in $\cC_K(\bY)$, which 
easily follows from Lemma \ref{complex_lipschitz}).
 
 We record it as a theorem. A similar statement 
 appears as \cite[Theorem H]{BBF}, which is weaker since 
 it is only stated  that $f$ is WPD
 in $\cC_K(\bY)$, but not 
 the acylindricity of the action.
 
 \begin{thm}
 Make the same assumptions as in Theorem \ref{qtree.acyl}.
 Then, for a sufficiently large $K$,
 $\cC_K(\bY)$ is a quasi-tree on which $G$ acts acylindrically such that 
 the given hyperbolic, WPD element $f$ with an axis $\gamma$,
 is hyperbolic with an axis $\gamma$ in $\cC_K(\bY)$.
 \end{thm}

\bibliographystyle{alpha}
\bibliography{./ref}

\end{document}